\documentclass[12pt]{amsart}
\usepackage{amsmath}
\usepackage{amscd}
\usepackage{amssymb}
\usepackage[all]{xy}
\usepackage[dvips]{graphicx}
\usepackage[cp1251]{inputenc}
\usepackage[T2A]{fontenc}
\usepackage[margin=3.25cm]{geometry}

\newtheorem{theorem}{Theorem}
\newtheorem{lemma}{Lemma}

\newtheorem{definition}{Definition}

\newcommand\Mman{M}
\newcommand\Vman{V}

\newcommand\Wman{W}
\newcommand\eps{\varepsilon}
\newcommand\Psp{P}

\newcommand\RRR{{\mathbb R}}

\newcommand\FF{{\mathcal F}}
\newcommand\Cinf{C^{\infty}}
\newcommand\Cont[1]{\mathcal{C}^{#1}}
\newcommand\Cr[3]{\Cont{#1}(#2,#3)}
\newcommand\Ci[2]{\Cr{\infty}{#1}{#2}}

\newcommand\Int{\mathrm{Int}}
\newcommand\id{\mathrm{id}}

\newcommand\func{f}
\newcommand\gfunc{g}
\newcommand\dif{h}

\newcommand\FMPa{\FF'(\Mman,\Psp)}

\newcommand\FMSa{\FF'(\Mman, S^1)}

\newcommand\FMPl{\FF_{l.c.}(\Mman,\Psp)}

\newcommand\FMSl{\FF_{l.c.}(\Mman,S^1)}

\newcommand\FMSc{\FF_{cov}(\Mman, S^1)}
\newcommand\FMRxi{\FF_{\xi}(\Mman,I)}
\newcommand\FMprRxi{\FF_{\xi}(\Mman',I)}

\newcommand\CMS{\Ci{\Mman}{S^1}}

\newcommand\hfunc{\widehat\func}
\newcommand\hgfunc{\widehat\gfunc}
\newcommand\hMman{\widehat\Mman}

\newcommand\KRf{\Gamma_{\func}}
\newcommand\KRg{\Gamma_{\gfunc}}

\begin{document}

\author{Sergiy Maksymenko}
\title[Circle-valued Morse functions on surfaces]
{Deformations of circle-valued Morse functions on surfaces}
\address{Topology dept., Institute of Mathematics of NAS of Ukraine, Te\-re\-shchenkivska st. 3, Kyiv, 01601 Ukraine}
\email{maks@imath.kiev.ua}
\date{09.06.2010}


\thanks{This research is partially supported by grant of Ministry of Science and Education of Ukraine, No. M/150-2009}

\begin{abstract}
Let $M$ be a smooth connected orientable compact surface.
Denote by $\mathcal{F}_{cov}(M, S^1)$ the space of all Morse functions $f:M\to S^1$ having no critical points on $\partial M$ and such that for every connected component $V$ of $\partial M$, the restriction $f:V \to S^1$ is either a constant map or a covering map.
Endow $\mathcal{F}_{cov}(M, S^1)$ with $C^{\infty}$-topology.
In this note the connected components of $\mathcal{F}_{cov}(M, S^1)$ are classified.
This result extends the results of S.~V.~Matveev, V.~V.~Sharko, and the author for the case of Morse functions being locally constant on $\partial M$.
\end{abstract}

\maketitle

\section{Introduction}

Let $\Mman$ be a compact surface and $\Psp$ be either the real line $\RRR$ or the circle $S^1$.
Denote by $\FMPa$ the subset of $\CMS$ consisting of maps $\func:\Mman\to\Psp$ such that 
\begin{enumerate}
 \item[(1)]
all critical points of $\func$ are non-degenerate and belongs to the interior of $\Mman$, so $\func$ is a \emph{Morse}.
\end{enumerate}

Let also $\FMPl$ be the subset of $\FMPa$ consisting of maps $\func:\Mman\to\Psp$ such that 
\begin{enumerate}
 \item[(2)]
$\func|_{\partial\Mman}$ is a \emph{locally constant} map, that is for every connected component $\Vman$ of $\partial\Mman$ the restriction of $\func$ to $\Vman$ is a constant map.
\end{enumerate}

Moreover, for the case $\Psp=S^1$ let $\FMSc$ be another subset of $\FMSa$ consisting of maps $\func:\Mman\to S^1$ such that 
\begin{enumerate}
 \item[(2$'$)]
for every connected component $\Vman$ of $\partial\Mman$ the restriction of $\func$ to $\Vman$ is either a \emph{constant} map or a \emph{covering} map.
\end{enumerate}
Thus 
$$
\FMSl \subset \FMSc.
$$
Endow all these spaces $\FMPa$, $\FMPl$, and $\FMSc$ with the corresponding $\Cinf$-topologies.
The connected components of the spaces $\FMPl$ were described in~\cite{Sharko:InstMath:1998, Kudryavtseva:MatSb:1999, Maks:InstMath:1998, Maks:CMH:2005}.
The aim of this note is to describe the connected components of the space $\FMSc$ for the case when $\Mman$ is orientable.

To formulate the result fix an orientation of $\Psp$ and let $\func\in\FMPa$.
Then for each (non-degenerate) critical point of $\func$ we can define its index with respect to a given orientation of $S^1$.
Denote by $c_i=c_i(\func)$, $i=0,1,2$, the total number of critical points of $\func$ of index $i$.

Moreover, suppose $\Vman$ is a connected component of $\partial\Mman$ such that the restriction of $\func$ to $\Vman$ is a constant map.
Then we associate to $\Vman$ the number $\eps_{\Vman}(\func):=+1$ (resp. $\eps_{k}(\func):=-1$) whenever the value $\func(\Vman)$ is a local maximum (resp. minimum) with respect to the orientation of $\Psp$.
If $\func|_{\Vman}$ is non-constant, then we put $\eps_{\Vman}(\func)=0$.

The following theorem describes the connected components of $\FMPl$.
\begin{theorem}\label{th:pi0_FMPl}
\cite{Sharko:InstMath:1998, Kudryavtseva:MatSb:1999, Maks:InstMath:1998, Maks:CMH:2005}
Let $\func,\gfunc\in\FMPl$.
Then they belong to the same path component of $\FMPl$ iff the following three conditions hold true:
\begin{enumerate}
 \item[\rm(i)]
$\func$ and $\gfunc$ are homotopic as continuous maps (for the case $\Psp=\RRR$ this condition is, of course, trivial);
 \item[\rm(ii)]
$c_i(\func)=c_i(\gfunc)$ for $i=0,1,2$;
 \item[\rm(iii)]
$\eps_{\Vman}(\func)=\eps_{\Vman}(\gfunc)$ for every connected component $\Vman$ of $\partial\Mman$.
\end{enumerate}
If $\Psp=\RRR$ and $\func=\gfunc$ on some neighbourhood of $\partial\Mman$, then one can choose a homotopy between $\func$ and $\gfunc$ fixed near $\partial\Mman$.
\end{theorem}
The case $\Psp=\RRR$ was independently established by V.~V.~Sharko~\cite{Sharko:InstMath:1998} and S.~V.~Matveev.
Matveev's proof was generalized to the case of height functions and published in the paper~\cite{Kudryavtseva:MatSb:1999} by E.~Kudryavtseva.
The case $\Psp=S^1$ was proven by the author in~\cite{Maks:InstMath:1998}.
Moreover, in~\cite{Maks:CMH:2005} Theorem~\ref{th:pi0_FMPl} was reproved by another methods.

The present notes establishes the following result:
\begin{theorem}\label{th:pi0_FMSc}
Suppose $\Mman$ is orientable.
Let $\func,\gfunc\in\FMSc$.
Then they belong to the same path component of $\FMSc$ iff the following three conditions hold true:
\begin{enumerate}
 \item[\rm(i)]
$\func$ and $\gfunc$ are homotopic as continuous maps;
 \item[\rm(ii)]
$c_i(\func)=c_i(\gfunc)$ for $i=0,1,2$;
 \item[\rm(iii)]
$\eps_{\Vman}(\func)=\eps_{\Vman}(\gfunc)$ for every connected component $\Vman$ of $\partial\Mman$ such that $\func|_{\Vman}$ is a constant map.
\end{enumerate}
\end{theorem}
Notice that the formulations of both Theorems~\ref{th:pi0_FMPl} and~\ref{th:pi0_FMSc} look the same.
The difference is that in Theorem~\ref{th:pi0_FMPl} every $\func\in\FMPl$ takes constant values of connected components of $\partial\Vman$, while in Theorem~\ref{th:pi0_FMSc} the restrictions of $\func\in\FMSc$ to boundary components $\Vman$ of $\Mman$ may also be covering maps and the degrees of such restrictions $\func:\Vman\to S^1$ are encoded by homotopy condition (i).

I was asked about the problem of connected components of the space $\FMSc$ by A.~Pajitnov.

The proof of Theorem~\ref{th:pi0_FMSc} follows the line of~\cite{Maks:InstMath:1998, Maks:CMH:2005}.
First we prove $\RRR$-variant of Theorem~\ref{th:pi0_FMSc} similarly to~\cite{Maks:CMH:2005}, see Theorem~\ref{th:pi0_FMRxi} below, and then deduce Theorem~\ref{th:pi0_FMSc} from Theorem~\ref{th:pi0_FMRxi} similarly to~\cite{Maks:InstMath:1998}.
Therefore we mostly sketch the proofs indicating only the principal differences.

\section{$\RRR$-variant of Theorem~\ref{th:pi0_FMSc} for surfaces with corners}
Let $\func\in\FMSc$.
Say that $v\in S^1$ is an \emph{exceptional} value of $\func$, if $v$ is either critical value of $\func$ or there exists a connected component $\Vman$ of $\partial\Mman$ such that $\func(\Vman)=v$.

Let $v\in S^1$ be a non-exceptional value of $\func$.
Then its inverse image $\func^{-1}(v)$ is a proper $1$-submanifold of $\Mman$ which does not contain connected components of $\partial\Mman$.
Thus $\func^{-1}(v)$ is a disjoint union of circles and arcs with ends on $\partial\Mman$ and transversal to $\partial\Mman$ at these points.
Let $\hMman$ be a surface obtained by cutting $\Mman$ along $\func^{-1}(v)$.

Then $\hMman$ can be regarded as a surface with corners and $\func$ induces a function $\hfunc:\hMman\to[0,1]$ such that
\begin{enumerate}
 \item[(a)]
$\hfunc|_{\Int{\hMman}}$ is Morse and has no critical points on $\partial\hMman$;
 \item[(b)]
let $\Wman$ be a connected component of $\partial\hMman$;
then either $\hfunc|_{\Wman}$ is constant, or there are $4k_{\Wman}$ points on $\Wman$ for some $k_{\Wman}\geq1$ dividing $\Wman$ into $4k_{\Wman}$ arcs
$$ 
A_1, B_1, C_1, D_1, \ldots, A_{k_{\Wman}}, B_{k_{\Wman}}, C_{k_{\Wman}}, D_{k_{\Wman}}
$$ 
such that $\hfunc$ strictly decreases on $A_i$, strictly increases on $C_i$, $\hfunc(B_i)=1$, and $\hfunc(D_i)=0$ for each $i=1,\ldots,k_{\Wman}$.
\end{enumerate}

We will now define the space of all such functions and describe its connected components.

\subsection{Space $\FMRxi$.}
Let $\Mman$ be a compact, possibly non-connected, surface.
For every connected component $\Wman$ of $\partial\Mman$ fix an orientation and a number $k_{\Vman}\geq0$, and divide $\Wman$ into $4k_{\Vman}$ consecutive arcs
$$
A_1, B_1, C_1, D_1, \ldots, A_{k_{\Wman}}, B_{k_{\Wman}}, C_{k_{\Wman}}, D_{k_{\Wman}}
$$
directed along orientation of $\Wman$.
If $k_{\Vman}=0$ then we do not divide $\Wman$ at all.

Denote this subdivision of $\partial\Mman$ by $\xi$ and the set of ends of these arks by $K=K(\xi)$.
We will regard $K$ as ``corners'' of $\Mman$.

Let also $T_{+}$ (resp. $T_{1}$, $T_{-}$, and $T_{0}$) the union of all closed arcs $A_i$ (resp. $B_i$, $C_i$, $D_i$) over all boundary components of $\Mman$.

Let $\FMRxi$ be the space of all continuous functions $\func:\Mman\to I=[0,1]$ satisfying the following three conditions.
\begin{enumerate}
 \item[(a)]
The restriction of $\func$ to $\Mman\setminus K$ is $\Cinf$,
and all partial derivatives of $\func$ of all orders continuously extends to all of $\Mman$.
 \item[(b)]
All critical points of $\func$ are non-degenerate and belong to $\Int{\Mman}$,
$$
\func(\Int{\Mman})\subset(0,1),
\quad
\func^{-1}(0)=T_{0},
\quad
\func^{-1}(1)=T_{1},
$$
and $\func|_{T_{+}}$ (resp. $\func|_{T_{-}}$) has strictly positive (resp. negative) derivative.
 \item[(c)]
Let $\Wman$ be a connected component of $\partial\Mman$ such that $k_{\Wman}=0$.
Then $\func|_{\Wman}$ is constant and $\hfunc(\Wman)\in(0,1)$.
\end{enumerate}

Notice that condition (a) means that $\func$ is a $\Cinf$-function on a surface with corners and condition (b) implies that $\func$ strictly increases (decreases) on each arc $A_i$ ($C_i$),

Again we associate to every $\func\in\FMRxi$ the total number $c_i(\func)$ of critical points at each index $i=0,1,2$.
Moreover, to every connected component $\Wman$ of $\partial\Mman$ with $k_{\Wman}=0$ we associate the number $\eps_{\Wman}(\func)$ as above.

The following theorem extends $\RRR$-case of Theorem\;\ref{th:pi0_FMPl} to orientable surfaces with corners.
\begin{theorem}\label{th:pi0_FMRxi}
Suppose $\Mman$ is orientable and connected.
Then $\func,\gfunc\in\FMRxi$ belongs to the same path component of $\FMRxi$ iff 
\begin{enumerate}
 \item[\rm(i)]
$c_i(\func)=c_i(\gfunc)$ for $i=0,1,2$;
 \item[\rm(ii)]
$\eps_{\Wman}(\func)=\eps_{\Wman}(\gfunc)$ for every connected component $\Wman$ of $\partial\Mman$ with $k_{\Wman}=0$.
\end{enumerate}
Moreover, if $\func=\gfunc$ on some neighbourhood of $T_{0}\cup T_{1}$, then there exists a homotopy relatively $T_{0}\cup T_{1}$ between these functions in $\FMRxi$.
\end{theorem}
The proof will be given in~\S\ref{sect:proof_th:pi0_FMSxi}.
Now we will deduce from this result Theorem~\ref{th:pi0_FMSc}.

\section{Proof of Theorem~\ref{th:pi0_FMSc}}
Necessity is obvious, therefore we will prove only sufficiency.

Let $\func,\gfunc\in\FMSc$.
Consider the following conditions $(Q_n)$, $n\geq0$, and $(P)$ for $\func$ and $\gfunc$.

\begin{enumerate}
 \item[$(P_n)$]
$\func$ (resp. $\gfunc$) is homotopic in $\FMSc$ to a map $\tilde\func$ (resp. $\tilde\gfunc$) such that for some common non-exceptional value $v\in S^1$ of $\tilde\func$ and $\tilde\gfunc$ the intersection $\tilde\func^{-1}(v) \cap \tilde\gfunc^{-1}(v)$ is transversal and consists of at most $n$ points.
 \item[$(Q)$]
$\func$ (resp. $\gfunc$) is homotopic in $\FMSc$ to a map $\tilde\func$ (resp. $\tilde\gfunc$) such that for some common non-exceptional value $v\in S^1$ of $\tilde\func$ and $\tilde\gfunc$, 
\begin{itemize}
 \item[(i)]
$\tilde\func^{-1}(v)=\tilde\gfunc^{-1}(v)$, 
 \item[(ii)]
$\tilde\func=\tilde\gfunc$ on some neighbourhood of $\tilde\func^{-1}(v)$, 
 \item[(iii)]
and for every connected component $\Mman_1$ of $\Mman\setminus\tilde\func^{-1}(v)$ the restrictions $\tilde\func$ and $\tilde\gfunc$ onto $\Mman_1$ have the same numbers of critical points at each index.
\end{itemize}
\item[$(R)$]
$\func$ is homotopic to $\gfunc$ in $\FMSc$.
\end{enumerate}

Notice that $\func$ and $\gfunc$ always satisfy $(P_n)$ for some $n\geq0$.
We have to prove for them condition $(R)$.
This is given by the following lemma, which completes the proof of Theorem~\ref{th:pi0_FMSc}.
\begin{lemma}\label{lm:deduce_pi0_FMSc_from_FMSxi}
Let $\func,\gfunc\in\FMSc$.
Suppose that $\func,\gfunc\in\FMSc$ satisfy conditions {\rm(i)-(iii)} of Theorem\;\ref{th:pi0_FMSc}.
Then the following implications hold:
$$
(P_n) \Rightarrow (P_{n-1}) \Rightarrow \cdots \Rightarrow (P_0) \Rightarrow (Q) \Rightarrow (R).
$$
\end{lemma}
\begin{proof}
Implications $(P_n) \Rightarrow (P_{n-1})$ and $(P_0) \Rightarrow (Q)$ can be deduced from Theorem~\ref{th:pi0_FMRxi} almost by the same arguments as \cite[Theorems~3\&4]{Maks:InstMath:1998} were deduced from the $\RRR$-case of Theorem~\ref{th:pi0_FMPl}.
The principal difference here is that one should work with $1$-submanifolds with boundary rather than with closed $1$-submanifolds.
The proof is left for the reader.

$(Q) \Rightarrow (R)$.
Cut $\Mman$ along $\func^{-1}(v)$ and denote the obtained surface with corners by $\hMman$.
Then $\func$ (resp. $\gfunc$) induces on $\hMman$ a function $\hfunc$ (resp. $\hgfunc$) belonging to $\FMprRxi$.
Moreover, it follows from conditions (i)-(iii) of Theorem~\ref{th:pi0_FMPl} for $\func$ and $\gfunc$ and assumption (iii) of $(Q)$ that for every connected component $\Mman_1$ of $\hMman$, the restrictions of $\hfunc$ and $\hgfunc$ to $\Mman_1$ satisfy conditions (i) and (ii) of Theorem~\ref{th:pi0_FMRxi}.
Hence they are homotopic in $\FMprRxi$ relatively some neighbourhood of the set $T_0\cup T_1$ corresponding to $\func^{-1}(v)$.
This homotopy yields a desired homotopy between $\func$ and $\gfunc$ in $\FMSc$.
\end{proof}

\section{Proof of Theorem~\ref{th:pi0_FMRxi}}\label{sect:proof_th:pi0_FMSxi}
We will follow the line of the proof of Theorem~\ref{th:pi0_FMPl}, see~\cite{Kudryavtseva:MatSb:1999,Maks:CMH:2005}.
Suppose $\func,\gfunc\in\FMRxi$ satisfy assumptions (i) and (ii) of Theorem~\ref{th:pi0_FMRxi}.
The idea is to reduce the situation to the case when $\gfunc=\func\circ\dif$ for some diffeomorphism $\dif$ of $\Mman$ fixed near $\partial\Mman$, and then show that $\func\circ\dif$ is homotopic in $\FMRxi$ to $\func$, see Lemmas~\ref{lm:f_is_homotopic_to_a_canonical_func}-\ref{lm:f_is_homotopic_to_fh}.
\subsection{KR-graph}
For $\func\in\FMRxi$ define the \emph{Kronrod-Reeb} graph (or simply \emph{KR-graph}) $\KRf$ of $\func$ as a topological space obtained by shrinking to a point every connected component of $\func^{-1}(v)$ for each $v\in I$.
It easily follows from the assumptions on $\func$ that $\KRf$ has a natural structure of a $1$-dimensional CW-complex.
The vertices of $\func$ corresponds to the connected components of level sets $\func^{-1}(v)$ containing critical points of $\func$.

Notice that $\func$ can be represented as the following composite of maps:
$$
\func: \func_{KR} \circ p_{\func}: 
\Mman \xrightarrow{~p_{\func}~} \KRf \xrightarrow{~\func_{KR}~} I,
$$
where $p_{\func}$ is a factor map and $\func_{KR}$ is the induced function on $\KRf$ which we will call the \emph{KR-function} of $\func$.

Say that $\func$ is \emph{generic} if it takes distinct values at distinct critical points and connected components $\Wman$ of $\partial\Mman$ with $k_{\Wman}=0$.
It is easy to show that every $\func\in\FMRxi$ is homotopic in $\FMRxi$ to a generic function.

Notice that for each non-exceptional value $v$ of $\func$ every connected component $P$ of $\func^{-1}(v)$ is either an arc or a circle.
We will distinguish the corresponding points on $\KRf$ as follows: if $P$ is an arc, then we denote the corresponding point on $\KRf$ in bold.
Thus on the KR-graph of $\func$ we will have two types of edges \emph{bold} and \emph{thin}.

Moreover, every vertex $w$ of degree $1$ of $\KRf$ corresponds either to a local extreme of $\func$ or to a boundary component $\Wman$ of $\partial\Mman$ with $k_{\Wman}=0$.
In the first case $w$ will be called an \emph{$e$-vertex}, and a \emph{$\partial$-vertex} otherwise.
$\partial$-vertexes will be denoted in bold.

Possible types of vertexes of $\KRf$ corresponding to saddle critical points together with the corresponding critical level sets are shown in Figure~\ref{fig:vertexes_of_KRf}.
\begin{figure}[ht]
\centerline{
\begin{tabular}{ccccc}
\includegraphics[height=1.2cm]{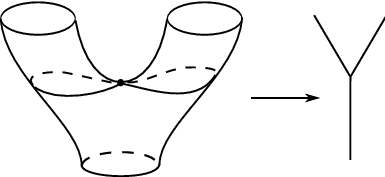} & \qquad\qquad &
\includegraphics[height=1.2cm]{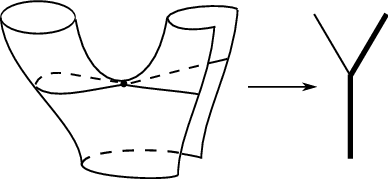} & \qquad\qquad &
\includegraphics[height=1.2cm]{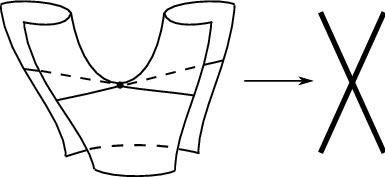} \\
a) & & b) & & c)
\end{tabular}
}
\protect\caption{}\label{fig:vertexes_of_KRf}
\end{figure}

\begin{definition}\label{defn:equivalence_of_KR_func}
Let $\func,\gfunc\in\FMRxi$.
Say that KR-functions of $\func$ and $\gfunc$ are {\bfseries equivalent} if there exist a homeomorphism $H:\KRf\to\KRg$ between their KR-graphs and a homeomorphism $\Phi:I\to I$ which preserves orientation such that $\gfunc_{KR}=\Phi^{-1}\circ\func_{KR}\circ H$ and $H$ maps bold edges (resp. thin edges, $\partial$-vertexes) of $\KRg$ to bold edges (resp. thin edges, $\partial$-vertexes) of $\KRf$.
\end{definition}

We will always draw a KR-graph so that the corresponding KR-function will be the projection to th evertical line.
This determines KR-function up to equivalence in the sense of Definition~\ref{defn:equivalence_of_KR_func}.

The following statement can be proved similarly to~\cite{Kulinich:MFAT:1998, BolsinovFomenko:1997}.
\begin{lemma}\label{lm:KR-equivalence-of-KR-functions}
Suppose $\Mman$ is orientable, and let $\func,\gfunc\in\FMRxi$ be two generic functions such that their KR-functions are KR-equivalent.
Then there exist a diffeomorphism $\dif:\Mman\to\Mman$ and a preserving orientation diffeomorphism $\phi:I\to I$ such that $\gfunc=\phi^{-1}\circ\func\circ\dif$.

Since $\phi$ is isotopic to $\id_{I}$, it follows that $\gfunc$ is homotopic in $\FMRxi$ to $\func\circ\dif$.
\end{lemma}

\subsection{Canonical KR-graph}
Consider the graphs shown in Figure~\ref{fig:elementary_graphs}.

The graph $X^{0}(k)$, $k\geq1$, consists of a bold line ``intersected'' by another $k-1$ bold lines, the graph $X^{\pm}(k)$ is obtained from $X^{0}(k)$ by adding a thin edge directed either up or down.
The vertex of degree $1$ on that thin edge can be either $e$- or $\partial$-one.

The graph $Y$ is determined by five numbers: $z, b_{-}, b_{+}, e_{-}, e_{+}$, where $z$ is the total number of cycles in $Y$, $b_{-}$ (resp. $e_{-}$) is the total number of $\partial$-vertexes (resp. $e$-vertexes) being local minimums for the KR-function, and $b_{+}$ and $e_{+}$ correspond to local maximums.

We will assume that KR-function on $X^{*}(k)$ surjectively maps this graph onto $[0,1]$, while KR-function of $Y$ maps it into interval $(0,1)$.

\begin{figure}[ht]
\centerline{
\includegraphics[height=4cm]{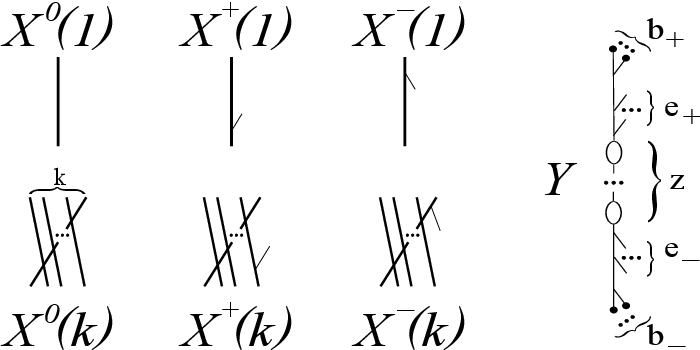}
}
\protect\caption{Elementary blocks of canonical KR-graphs}\label{fig:elementary_graphs}
\end{figure}

\begin{figure}[ht]
\centerline{
\includegraphics[height=2cm]{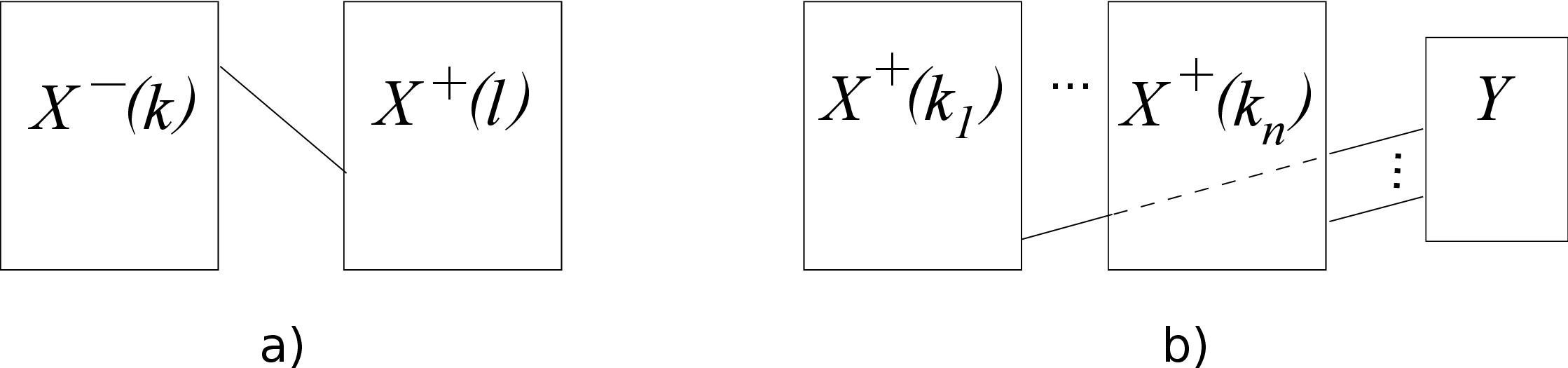}
}
\protect\caption{Canonical KR-graph $\KRf$}\label{fig:canonical_KR_graph}
\end{figure}

\begin{definition}
Let $\func\in\FMRxi$.
Say that $\func$ is {\bfseries canonical} if it is generic and its KR-graph $\KRf$ has one of the following forms:
\begin{enumerate}
 \item 
coincide either with one of $X^{*}(k)$ for some $k\geq1$, or with $Y$ for some $e_{\pm}$, $b_{\pm}$, and $z$;
 \item
is a union of $X^{-}(k)$ with $X^{+}(l)$ with common thin edge for some $k,l\geq1$, see Figure~\ref{fig:canonical_KR_graph}a);
 \item  
is a union of some $X^{+}(k_i)$, $i=1,\ldots,n$, connected along their thin edges with $Y$, see Figure~\ref{fig:canonical_KR_graph}b).
\end{enumerate}

Every maximal bold connected subgraph of $\KRf$ will be called an {\bfseries $X$-block}. Evidently, such a block is isomorphic with $X^{0}(k)$ for some $k$.
\end{definition}

\begin{lemma}\label{lm:canon-KR-graph_determ_crit_type}
Let $\func\in\FMRxi$ be a canonical function.
Then the numbers $c_i(\func)$, $k_{\Wman}$, and $\eps_{\Wman}(\func)$ are completely determined by its KR-graph $\KRf$ and wise verse.
Moreover, every $X$-block of $\KRf$ corresponds to a unique boundary component of $\Mman$.
In particular, the collection of $X$-blocks in $\KRf$ is determined (up to order) by the partition $\xi$ of $\partial\Mman$, and therefore does not depend on a canonical function $\func$.
\end{lemma}
\begin{proof}
Since $\func$ is generic, $c_0(\func)$ (resp. $c_2(\func)$) is equal to the total number of vertexes of degree $1$ being local minimums (resp. local maximums) of the restriction of $\func_{KR}$ to $Y$, while $c_1(\func)$ is equal to the total number of vertexes of $\KRf$ of degrees $3$ and $4$.

Furthermore, it easily follows from Figure~\ref{fig:vertexes_of_KRf}c) that every $X$-block $N$ of $\KRf$ corresponds to collar of some boundary component $\Wman$ of $\Mman$ such that $k_{\Wman}$ is equal to the total number of local minimums (=local maximums) of the restriction of $\func_{KR}$ to $N$.

Finally, every connected component $\Wman$ of $\partial\Mman$ with $k_{\Wman}=0$ corresponds to a $\partial$-vertex $w$ on $Y$.
Moreover, $\eps_{\Wman}=-1$ (resp. $\eps_{\Wman}=+1$) iff $w$ is a local minimum (resp. local maximum) of the restriction of $\func_{KR}$ to $Y$.
\end{proof}

\begin{lemma}\label{lm:f_is_homotopic_to_a_canonical_func}
Let $\func\in\FMRxi$.
Then $\func$ is homotopic in $\FMRxi$ to some canonical function.
\end{lemma}
\begin{proof}
Consider the following elementary surgeries of a KR-graph shown in Figure~\ref{fig:surgeries_of_KR_graph}.
\begin{figure}[ht]
\centerline{\includegraphics[height=3.5cm]{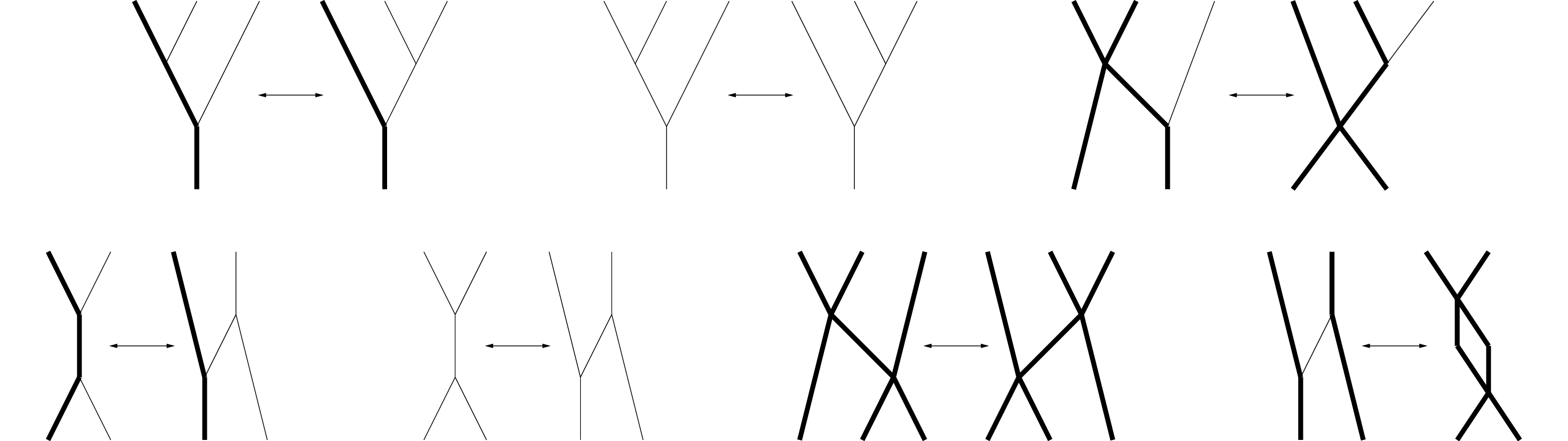}}
\caption{Elementary surgeries of KR-graph}\label{fig:surgeries_of_KR_graph}
\end{figure}
% \begin{figure}[ht]
% \centerline{
%  \includegraphics[height=1.1cm]{krg-01_a.eps}  \qquad 
%  \includegraphics[height=1.1cm]{krg-01.eps}  \qquad 
%  \includegraphics[height=1.1cm]{krg-03_a.eps}  \qquad
%  \includegraphics[height=1.1cm]{krg-03.eps}
% }
% 
% \smallskip
% 
% \centerline{
%  \includegraphics[height=1.1cm]{krg-02.eps} \qquad
%  \includegraphics[height=1.1cm]{krg-04.eps} \qquad 
%  \includegraphics[height=1.1cm]{krg-05.eps}
% }
% \caption{Elementary surgeries of KR-graph}\label{fig:surgeries_of_KR_graph}
% \end{figure}
It is easy to see that each of them can be realized by a deformation of $\func$ in $\FMRxi$.
Then similarly to~\cite[Lemma~11]{Kudryavtseva:MatSb:1999} one can reduce any KR-graph of $\func\in\FMRxi$ to a canonical form using these surgeries.
We leave the details for the reader.
\end{proof}

\begin{lemma}
Let $\func,\gfunc\in\FMRxi$ be two canonical functions satisfying assumptions {\rm(i)} and {\rm(ii)} of Theorem~\ref{th:pi0_FMRxi}.
Then $\func$ (resp. $\gfunc$) is homotopic in $\FMRxi$ to another canonical function $\tilde\func$ (resp. $\tilde\gfunc$) such that $\tilde\gfunc=\tilde\func\circ\dif$ for some diffeomorphism $\dif:\Mman\to\Mman$ fixed near $\partial\Mman$.
\end{lemma}
\begin{proof}
It follows from Lemma~\ref{lm:canon-KR-graph_determ_crit_type} and assumptions on $\func$ and $\gfunc$ that their KR-graphs have the same $Y$-blocks and the same (up to order) $X^{\pm}(k)$-blocks.
Then, using surgeries of Figure~\ref{fig:surgeries_of_KR_graph} applied to $\KRg$ we can reduce the situation to the case when KR-functions of $\func$ of $\gfunc$ are KR-equivalent.
Whence by Lemma~\ref{lm:KR-equivalence-of-KR-functions} we can also assume that there exists a diffeomorphism $\dif:\Mman\to\Mman$ such that $\gfunc=\func\circ\dif$.
Moreover, changing $\gfunc$ similarly to~\cite{Kudryavtseva:MatSb:1999} or~\cite{Maks:CMH:2005} one can choose $\dif$ so that it preserves orientation of $\Mman$, maps every connected component $\Wman$ of $\partial\Mman$ onto itself, and preserves subdivision $\xi$ on $\Wman$.
Then using the assumptions on $\func$ and $\gfunc$ near $\partial\Mman$, one can show that $\dif$ is isotopic to the identity near $\partial\Mman$.
\end{proof}

\begin{lemma}\label{lm:f_is_homotopic_to_fh}
Let $\dif:\Mman\to\Mman$ be a diffeomorphism fixed near $\partial\Mman$ and $\func\in\FMRxi$ be a canonical function.
Then $\func\circ\dif$ is homotopic in $\FMRxi$ to $\func$ relatively some neighbourhood of $\partial\Mman$.
\end{lemma}
\begin{proof}
Since every $X$-block of $\KRf$ corresponds to a collar $N(\Wman)$ of some boundary component $\Wman$ of $\partial\Mman$, we can assume that $\dif$ is fixed on some neighbourhood of $N(\Wman)$.
Therefore we may cut off $N(\Wman)$ from $\Mman$ and assume that $\func$ takes constant values at each boundary component of $\partial\Mman$.
Then $\func$ is homotopic to $\func\circ\dif$ relatively some neighbourhood of $\partial\Mman$ by the arguments similar to the proof of Theorem~\ref{th:pi0_FMPl}, see~\cite{Maks:CMH:2005}.
\end{proof}

Theorem~\ref{th:pi0_FMRxi} now follows from Lemmas~\ref{lm:f_is_homotopic_to_a_canonical_func}-\ref{lm:f_is_homotopic_to_fh}.

% \bibliographystyle{amsplain}
% \bibliographystyle{gost71s}
% \bibliography{biblio}

\end{document}